\documentclass[11pt,letterpaper]{amsart}
\usepackage[left=1.3in,top=1.2in,right=1.3in,bottom=1.25in]{geometry}
\usepackage{amsmath, amssymb,amsmath,amscd,amsfonts,amsthm,mathrsfs, graphicx}
\usepackage[arrow,matrix,curve,cmtip,ps]{xy}
\usepackage{enumitem}
\usepackage[alphabetic]{amsrefs}
\usepackage{pinlabel}
\usepackage{tikz-cd}
\usepackage{color}
\usepackage{xcolor}

\usepackage{centernot}

\usepackage{pinlabel}	

\newtheorem{proposition}{Proposition}[section]
\newtheorem{theorem}[proposition]{Theorem}
\newtheorem{corollary}[proposition]{Corollary}

\newtheorem{remark}[proposition]{Remark}

\newtheorem*{theorem*}{Theorem}
\newtheorem*{proposition*}{Proposition}
\newtheorem*{lemma*}{Lemma}
\newtheorem*{corollary*}{Corollary}
\newtheorem*{rep@theorem}{\rep@title}
\newcommand{\newreptheorem}[2]{
\newenvironment{rep#1}[1]{
 \def\rep@title{#2 \ref{##1}}
 \begin{rep@theorem}}
 {\end{rep@theorem}}}
\makeatother
\theoremstyle{definition}

\newreptheorem{theorem}{Theorem}
\newreptheorem{lemma}{Lemma}
\newreptheorem{proposition}{Proposition}
\newreptheorem{corollary}{Corollary}

\newcommand{\Q}{\mathbb{Q}}

\newcommand{\Z}{\mathbb{Z}}

\newcommand{\A}{\mathcal{A}}

\begin{document}
\title[Concordance to links with an unknotted component]{Concordance to links with an unknotted component}

\author{Christopher W.\ Davis}
\address{Department of Mathematics, University of Wisconsin--Eau Claire}
\email{daviscw@uwec.edu}
\urladdr{people.uwec.edu/daviscw}

\author{JungHwan Park}
\address{School of Mathematics, Georgia Institute of Technology}
\email{junghwan.park@math.gatech.edu }
\urladdr{people.math.gatech.edu/~jpark929/}

\date{\today}


\begin{abstract} 
We construct links of arbitrarily many components each component of which is slice and yet are not concordant to any link with even one unknotted component. The only tool we use comes from the Alexander modules.  
\end{abstract}

\maketitle
\section{Introduction}
In \cite{Cochran1991, CO1990, CO1993, CR2012}, Cochran, Cochran-Orr, and Cha-Ruberman proved variations of the following theorem:

\begin{theorem*}
There are links with slice components that are not concordant to any link with every component unknotted.\end{theorem*}

In \cite[Theorem 2.11]{Cochran1991}, Cochran used the $\beta^i$-invariants to show that there exist links $L_1\cup L_2$ with $L_1$ slice and $L_2$ unknotted which are not topologically concordant to any link with the first component unknotted. A similar result appears in \cite{CO1990, CO1993} using the complexity of a covering link. Further, in \cite[Theorem 1.1]{CR2012}, Cha-Ruberman used covering link calculus together with the correction term of Heegaard Floer homology to give topologically slice links $L_1\cup L_2$ with $L_1$ smoothly slice, $L_2$ unknotted, and which are not smoothly concordant to any link with the first component unknotted. 

All the examples above are links with the second component unknotted which are not concordant to any link with the first component unknotted. In this short note, we use a classical invariant to provide examples of links whose every component is slice but which satisfy the stronger conclusion that they are not concordant to any link with \emph{even one} unknotted component.  

\begin{theorem}\label{thm: 2-component}
The $2$-component link of Figure~\ref{fig: 2-component} has slice knots for its components but is not concordant to any link with an unknotted component.\end{theorem}

\begin{figure}[h]
\begin{picture}(150,95)
\put(0,5){\includegraphics[height=.15\textheight]{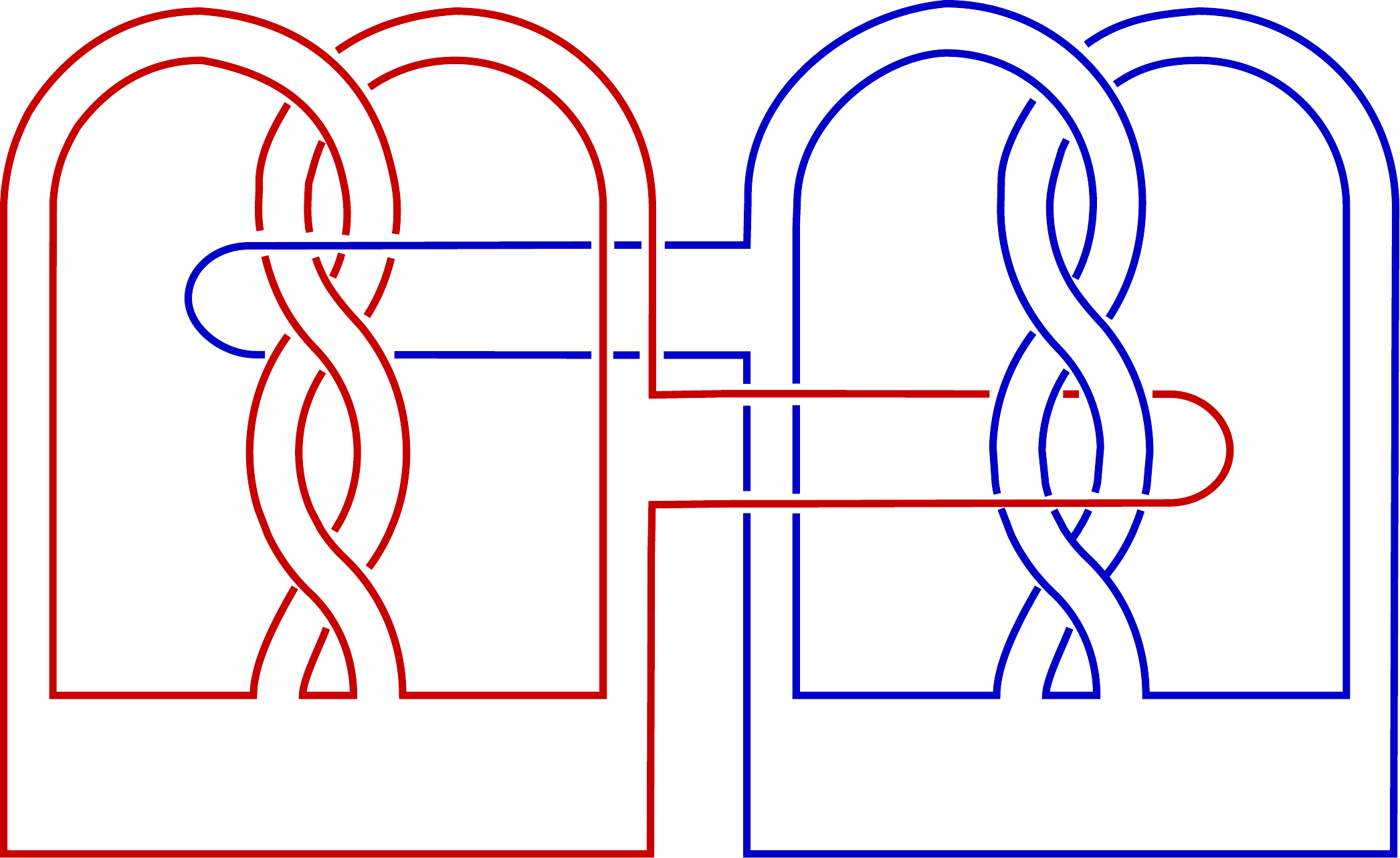}}
\end{picture}
\caption{A $2$-component link $L_1\cup L_2$ of Theorem~\ref{thm: 2-component}. }\label{fig: 2-component}
\end{figure} 

We generalize Theorem~\ref{thm: 2-component} to links of more than two components. We first state the case with $3$-components.

\begin{theorem}\label{thm: 3-component}
The $3$-component link of Figure~\ref{fig:L3} has slice links for its every proper sublink but is not concordant to any link with an unknotted component.\end{theorem}

\begin{figure}[h]
\begin{picture}(230,95)
\put(0,5){\includegraphics[height=.15\textheight]{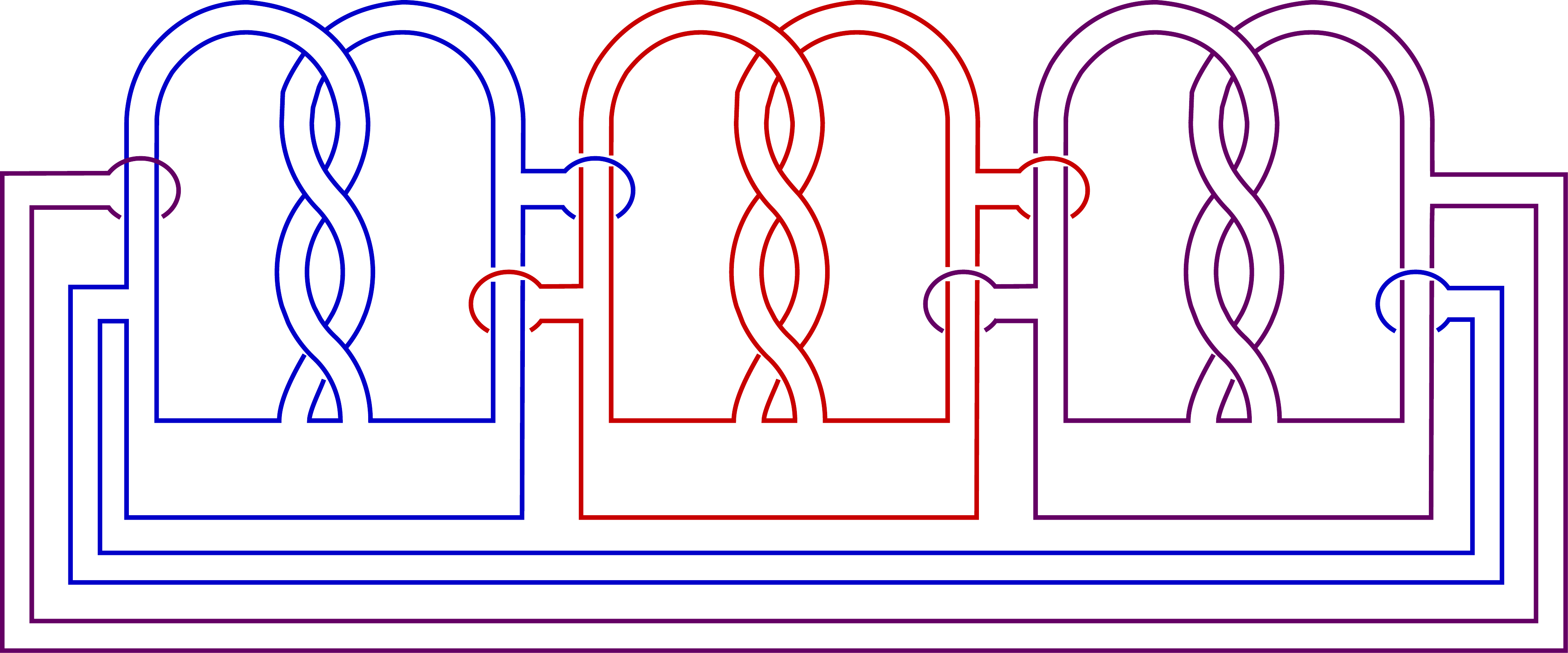}}
\end{picture}
\caption{A $3$-component link $L_1\cup L_2\cup L_3$ of Theorem~\ref{thm: 3-component}.}\label{fig:L3}
\end{figure} 

Theorem~\ref{thm: 3-component} is a special case of a the following more general result.

\begin{theorem}\label{thm: n-component}
For any $n\ge 3$, the $n$-component link of Figure~\ref{fig:Ln} has a slice link for its every $2$-component sublink but is not concordant to any link with an unknotted component. Moreover, its every proper sublink is concordant to a link with an unknotted component.\end{theorem}

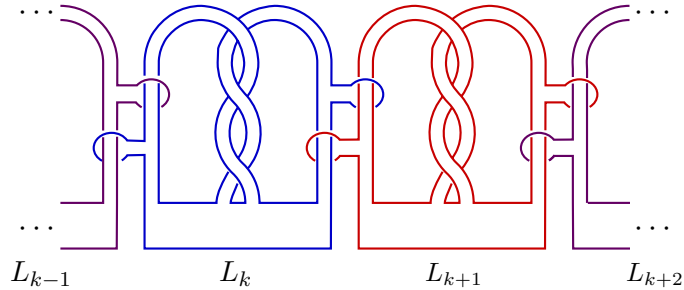
\begin{figure}[h]
\begin{picture}(240,95)
\put(-1,17){\large{$\cdots$}}
\put(-1,99){\large{$\cdots$}}
\put(-4,0){$L_{k-1}$}
\put(75,0){$L_{k}$}
\put(153,0){\small{$L_{k+1}$}}
\put(229,0){\small{$L_{k+2}$}}
\put(232,17){\large{$\cdots$}}
\put(232,99){\large{$\cdots$}}
\put(15,12){\includegraphics[height=.15\textheight]{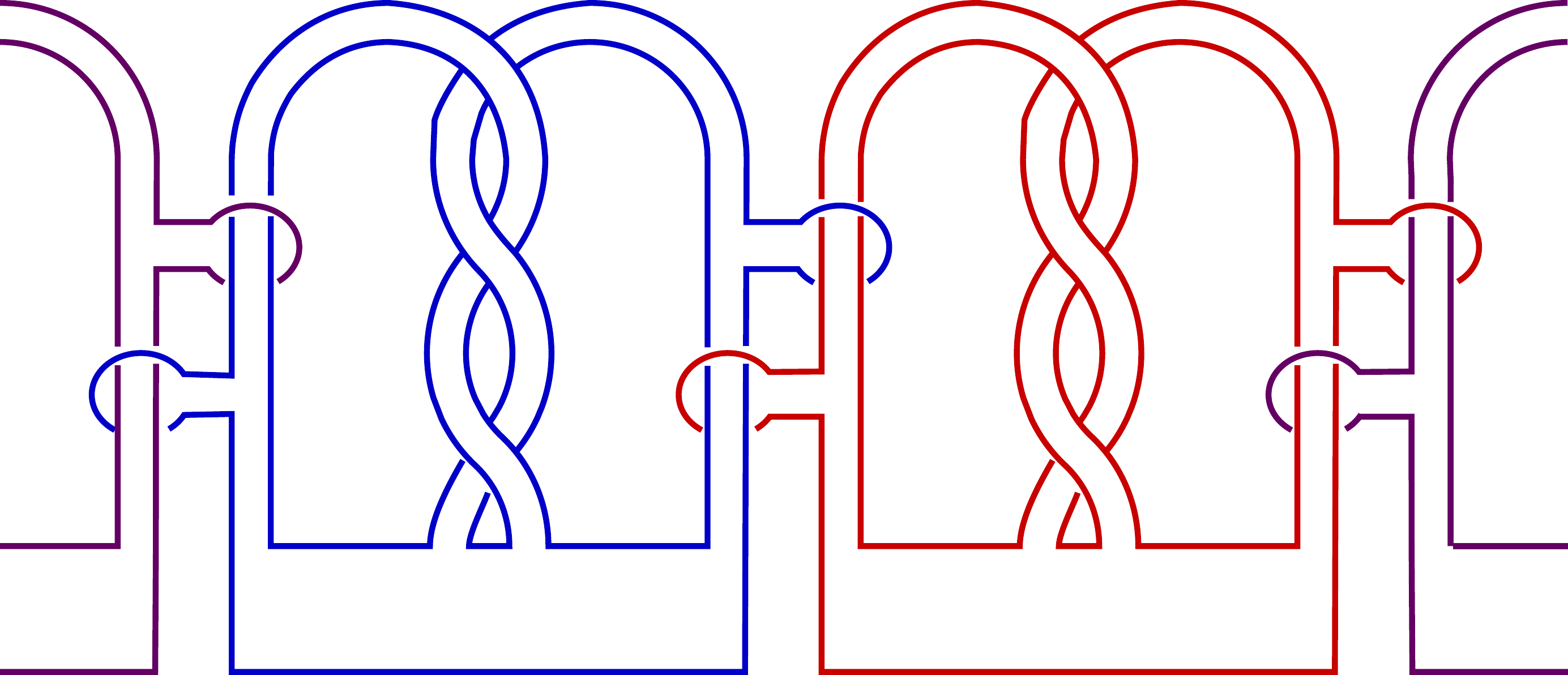}}
\end{picture}

\caption{An $n$-component link $L_1\cup \cdots\cup L_n$ with components indexed by $\Z/n\Z$ of Theorem~\ref{thm: n-component}. }\label{fig:Ln}
\end{figure} 

In fact, the preceding links are not concordant to any link that has a component with trivial Alexander polynomial. We extend this further by replacing trivial Alexander polynomial with any given finite collection of Alexander polynomials. This should be thought of it as a generalization of \cite[Theorem 1.3]{CR2012}.

\begin{theorem}\label{thm:alexanderpoly} For any finite collection $D$ of Alexander polynomials of knots and for any knot $J$ with $\Delta_{J}(t) \in D$, there are links $L = L_1 \cup L_2$ satisfying the following:
\begin{enumerate}
\item $L_1$ and $L_2$ are concordant to $J$.
\item $L$ is not concordant to any link $L' = L_1' \cup L_2'$ with either $\Delta_{L_1'}(t) \in D$ or $\Delta_{L_2'}(t) \in D$.
\end{enumerate}\end{theorem}

Remarkably, our obstruction comes from a classical invariant, the Alexander module, of a component of a link and the classes of the lifts of the remaining components. We recall the Alexander module and state the obstruction in Section~\ref{sect:tool}. In Section~\ref{sect:proof}, we use the obstruction to prove Theorems~\ref{thm: 2-component}, \ref{thm: 3-component}, \ref{thm: n-component}, and \ref{thm:alexanderpoly}. 

While the question of concordance to boundary links (as in \cite{CO1990,CO1993,Livingston1990}) is not our main focus we will take a moment and point out that the techniques of our paper produce links which are not concordant to boundary links (see Remark \ref{rk:boundary}). It is an interesting question to ask if our obstruction is related to Milnor's invariants.

This project is also motivated by the following question: does there exist a link in a homology sphere which is not concordant to any link in $S^3$, even when each component is concordant to a knot in $S^3$. Note that by performing $\frac{1}{p}$-surgery on a component of a link of Figure $3$, we get a new link where each component is concordant to a knot in $S^3$. We believe that this link is not concordant to any link in $S^3$, but we are not able to prove this at the moment. We also make a remark that the above question is a natural generalization of a theorem of Adam Levine \cite{Levine:2016-1} $($see also \cite{Hom-Levine-Lidman:2018-1}$)$, where he proved that there exists a knot in a homology sphere which is not smoothly concordant to any knot in $S^3$. As far as the authors knowledge, it is not known if such a statement is true for the topological category. 

\subsection*{Acknowledgments}
This project started when the first author was visiting the Georgia Institute of Technology. He thanks them for their support. We would also like to thank Lisa Piccirillo, Kouki Sato, Jennifer Hom, Kent Orr, Jae Choon Cha, Min Hoon Kim, and Mark Powell for helpful conversations.

\section{Obstruction: The Alexander module}\label{sect:tool}

For the rest of this paper, we work in the topological (locally flat) category. For any knot $K$, we denote by $E(K)$ the knot exterior $S^3 \setminus \nu(K)$, where $\nu(K)$ is an open tubular neighborhood of $K$. The first homology of the infinite cyclic cover of $E(K)$ with rational coefficients is a $\Q[t,t^{-1}]$-module, where the action of $t$ is induced by the deck transformation.  This module is the \emph{Alexander module} of  $K$ and is denoted by $\A(K)$. Similarly, if $D\subseteq B^4$ is a slice disk for $K$, then we denote by $E(D)$ the disk exterior $B^4 \setminus \nu(D)$, where $\nu(D)$ is an open tubular neighborhood of $D$. Again, the first homology of the infinite cyclic cover of $E(D)$ with rational coefficients is called the Alexander module of $D$ and denoted by $\A(D)$.

The Alexander module can be used to frame many obstructions to the sliceness of a knot. It is a well known fact that the Alexander module of a knot has a non-singular form called the Blanchfield form \cite{Blanchfield:1957-1} and if $K$ bounds a slice disk $D$, then the kernel of the map from $\A(K)$ to $\A(D)$ is a Lagrandian submodule \cite{Kearton:1975-2} with respect to the Blanchfield form. In particular, if $\A(K)$ is not the trivial module then this kernel cannot be all of $\A(K)$. Also, recall that a knot has trivial Alexander module if and only if it has trivial Alexander polynomial. Combining these facts we get the following well known result.
 
\begin{proposition}\label{prop: Alexander module slice disk}
If $K$ is a knot with nontrivial Alexander polynomial and $D$ is a slice disk for $K$, then $\A(K)\to \A(D)$ is not the zero homomorphism.\end{proposition}
 
 The following is the immediate corollary of Proposition~\ref{prop: Alexander module slice disk}
 
\begin{corollary}\label{cor:tool}
Let $L = L_1\cup \cdots \cup L_n$ be a link with vanishing pairwise linking numbers. Suppose $L_1$ is a slice knot with nontrivial Alexander polynomial and the classes of the lifts of $L_2, \ldots, L_n$ generate $\A(L_1)$. Then $L$ is not concordant to any link $L'=L_1'\cup\cdots\cup L_n'$ where $\Delta_{L_1}(t)$ and $\Delta_{L_1'}(t)$ are relatively prime. In particular, $L$ is not concordant to any link of the form $U\cup L_2' \cup \cdots\cup L_n'$ where $U$ is the unknot.\end{corollary}
 
\begin{proof} Suppose $L$ and $L'$ are concordant via $C = C_1\cup\cdots\cup C_n$.  Since $L_1$ is slice, $L_1'$ is slice as well.  Cap $C_1$ with a disk in the $4$-ball bounded by $L_1'$ to get a slice disk $D$ for $L_1$. Since $\Delta_{L_1}(t)$ is the annihilator of $\A(L_1)$, $$\Delta_{L_1}(t)\cdot [L_i]=0 \in \A(L_1) \text{ for every } i\in \{2,\ldots, n\}.$$  Here, $[L_i]$ indicates the class of the lift of $L_i$ to $\A(L_1)$. Similarly, $$\Delta_{L_1'}(t)\cdot [L_i']=0 \in \A(L_1') \text{ for every } i\in \{2,\ldots, n\}.$$  Since the lift of $L_i$ and the lift of $L'_i$ represent the same class in $\mathcal{A}(D)$, we have both of $\Delta_{L_1}(t)$ and $\Delta_{L_1'}(t)$ annihilating the classes of the lifts of $L_i$ in $\A(D)$.  Further, since $\Delta_{L_1}(t)$ and $\Delta_{L_1'}(t)$ are relatively prime, the classes of the lifts of $L_2,\ldots, L_n$ are trivial in $\A(D)$. This is not possible by Proposition~\ref{prop: Alexander module slice disk}.\end{proof}

\begin{remark}\label{rk:boundary}\normalfont Corollary~\ref{cor:tool} also gives an obstruction for links to be concordant to boundary links.  Indeed, if $L' = L'_1\cup \cdots \cup L'_n$ is a boundary link then by lifting Seifert surfaces for $L'_2,\ldots, L'_n$ to the infinite cyclic cover of $E(L'_1)$ we see that the classes of the lifts of $L'_2, \ldots, L'_n$ are trivial in $\A(L'_1)$.  It would be interesting to see if there exist links with vanishing Milnor's invariants which satisfy the hypotheses of Corollary~\ref{cor:tool}.\end{remark}

\section{proofs of Theorems~\ref{thm: 2-component}, \ref{thm: 3-component}, \ref{thm: n-component}, and \ref{thm:alexanderpoly}}\label{sect:proof}

We are now ready to prove that our examples satisfy the asserted conditions.

\begin{proof}[Proof of Theorem~\ref{thm: 2-component}] 

Let $L=L_1\cup L_2$ be the link in Figure~\ref{fig: 2-component}. Each of $L_1$ and $L_2$ is isotopic to the $9_{46}$ knot which is slice. The $9_{46}$ knot has a cyclic Alexander module $$\A(L_1) \cong \frac{\Q[t,t^{-1}]}{\langle(1-2t)\cdot(2-t)\rangle} $$ with a generator given by the lift of the curve depicted to the far right of Figure~\ref{fig: 2-component homotope}. Also, Figure~\ref{fig: 2-component homotope} describes a homotopy in the exterior of $L_1$ from $L_2$ to the curve whose lift generates $\A(L_1)$. 

\begin{figure}[h]
\begin{picture}(285,60)
\put(0,5){\includegraphics[height=.1\textheight]{L2.pdf}}
\put(43,-5) {$(a)$}
\put(120,5){\includegraphics[height=.1\textheight]{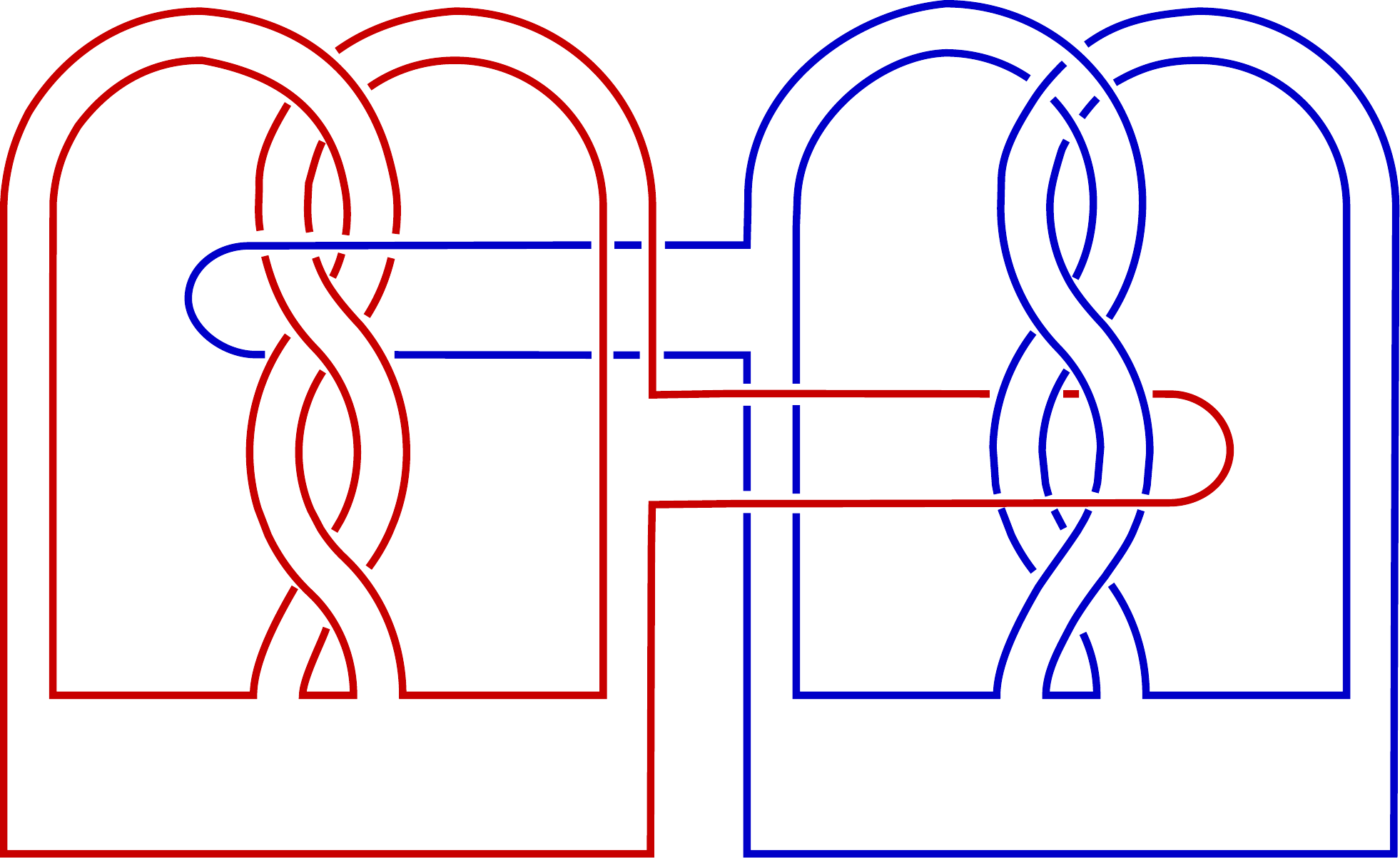}}
\put(163,-5) {$(b)$}
\put(240,5){\includegraphics[height=.1\textheight]{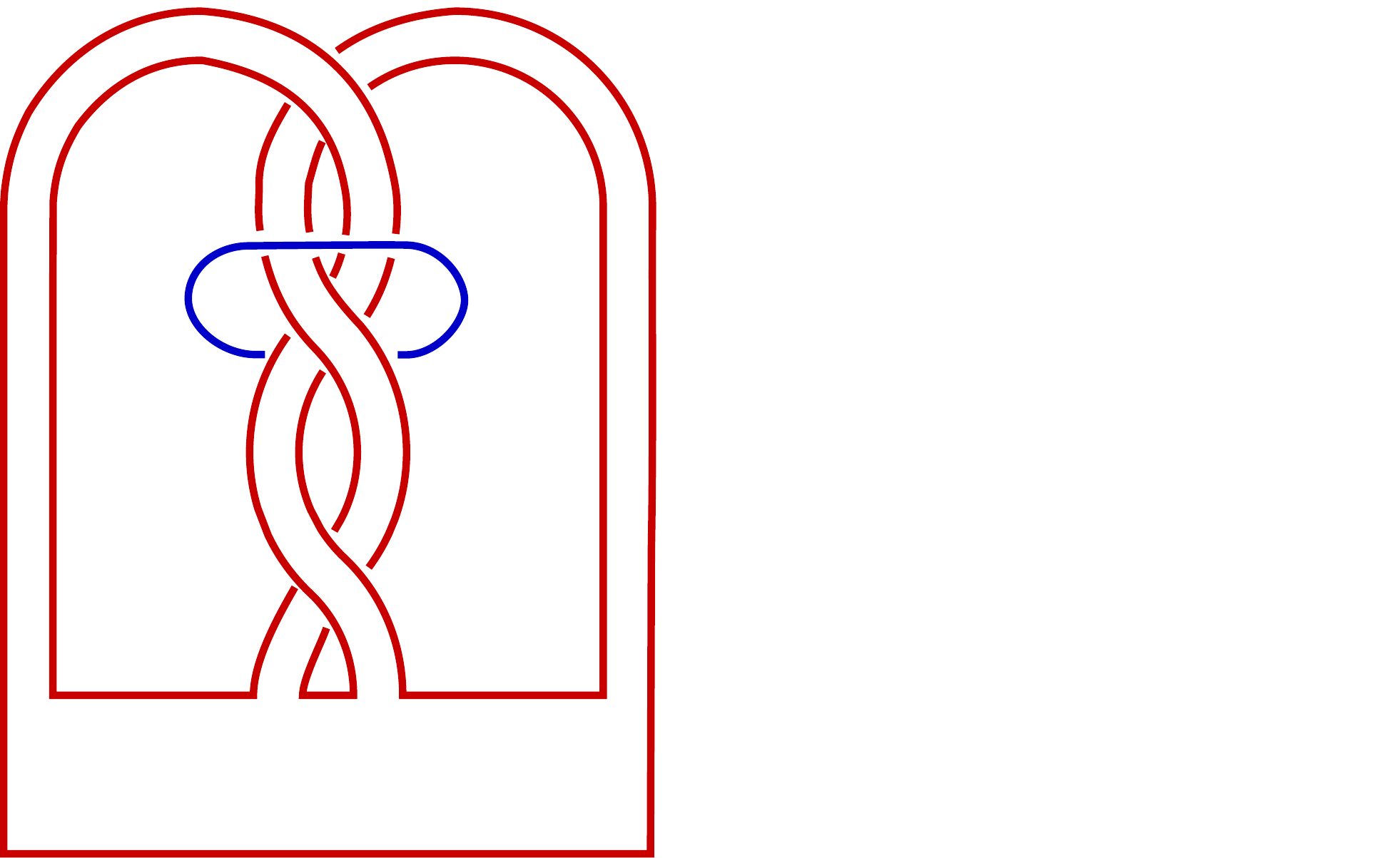}}
\put(257,-5) {$(c)$}
\end{picture}
\caption{Left to right:  (a) The Link $L_1\cup L_2$. (b) A homotopy changing 5 crossings of $L_2$ in the complement of $L_1$.  (c) An isotopy reduces the image of $L_2$ in (b) to a curve whose lift generates $\A(L_1)$.} \label{fig: 2-component homotope}
\end{figure} 

The class of the lift of $L_2$ generates $\A(L_1)$.  Thus, Corollary~\ref{cor:tool} concludes that $L$ is not concordant to any link $L_1'\cup L_2'$ with $L_1'$ unknotted. The proof is complete by the symmetry of $L$.\end{proof}

Since Theorem~\ref{thm: 3-component} is a special case of Theorem~\ref{thm: n-component}, we only prove Theorem~\ref{thm: n-component}.

\begin{proof}[Proof of Theorem~\ref{thm: n-component}.]
Let $L=L_1\cup L_2\cup\cdots\cup L_n$ be the link in Figure~\ref{fig:Ln}. As every component of $L$ is the $9_{46}$ knot, each component is slice. Further, every $2$-component sublink of $L$ is either isotopic to a link drawn in Figure~\ref{fig:3-comp sublink}~(a) or the split link $9_{46} \sqcup 9_{46}$. Observe that both links are slice, as shown in Figure~\ref{fig:3-comp sublink}. Let $L'$ be a proper sublink of $L$, then for some $k\in \Z/n$, $L_{k}$ is a component of $L'$ and $L_{k+1}$ is not. We may now modify $L'$ by changing $L_{k}$ by a similar band move to that depicted in Figure~\ref{fig:3-comp sublink}.  This reveals that $L'$ is concordant to a link with an unknotted component.

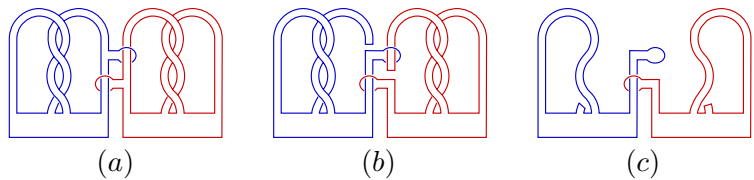
\begin{figure}[h]
\begin{picture}(290,65)
\put(0,0){\includegraphics[height=.1\textheight]{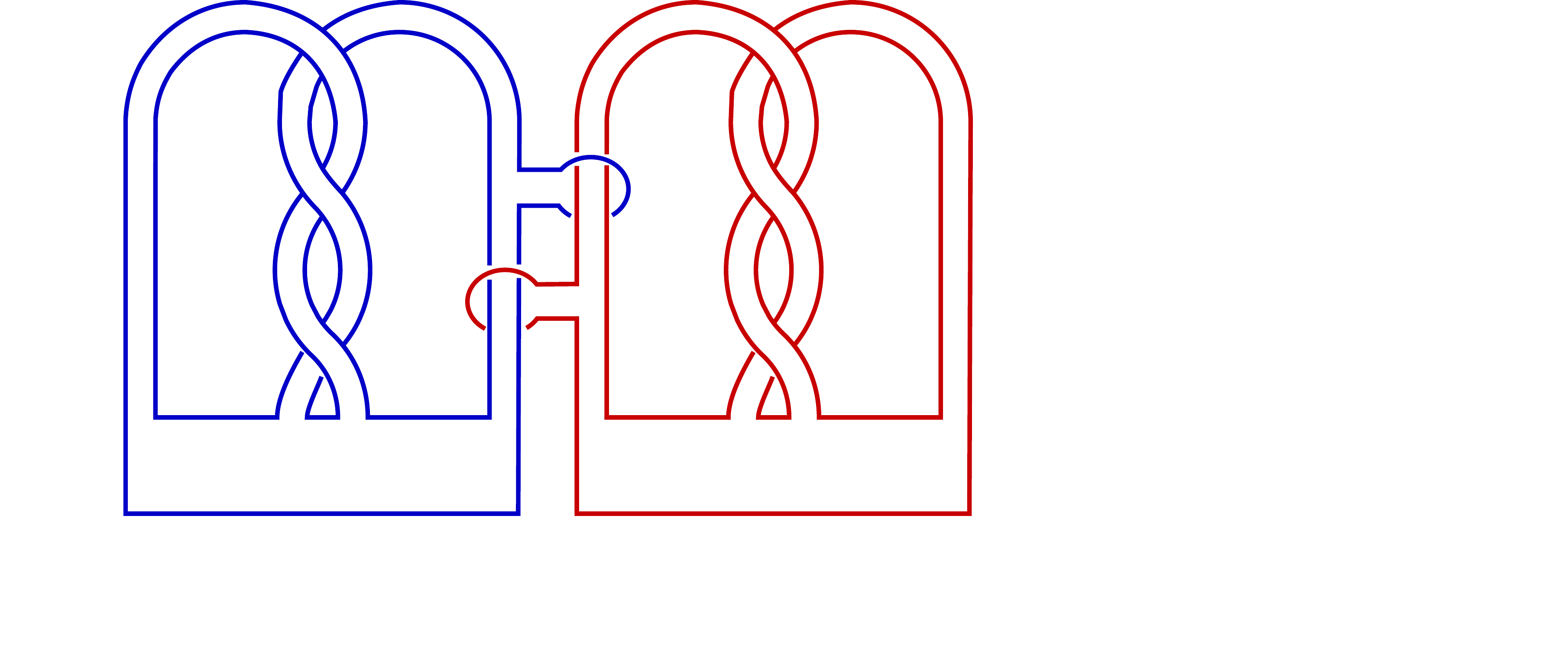}}
\put(45,0) {$(a)$}
\put(100,0){\includegraphics[height=.1\textheight]{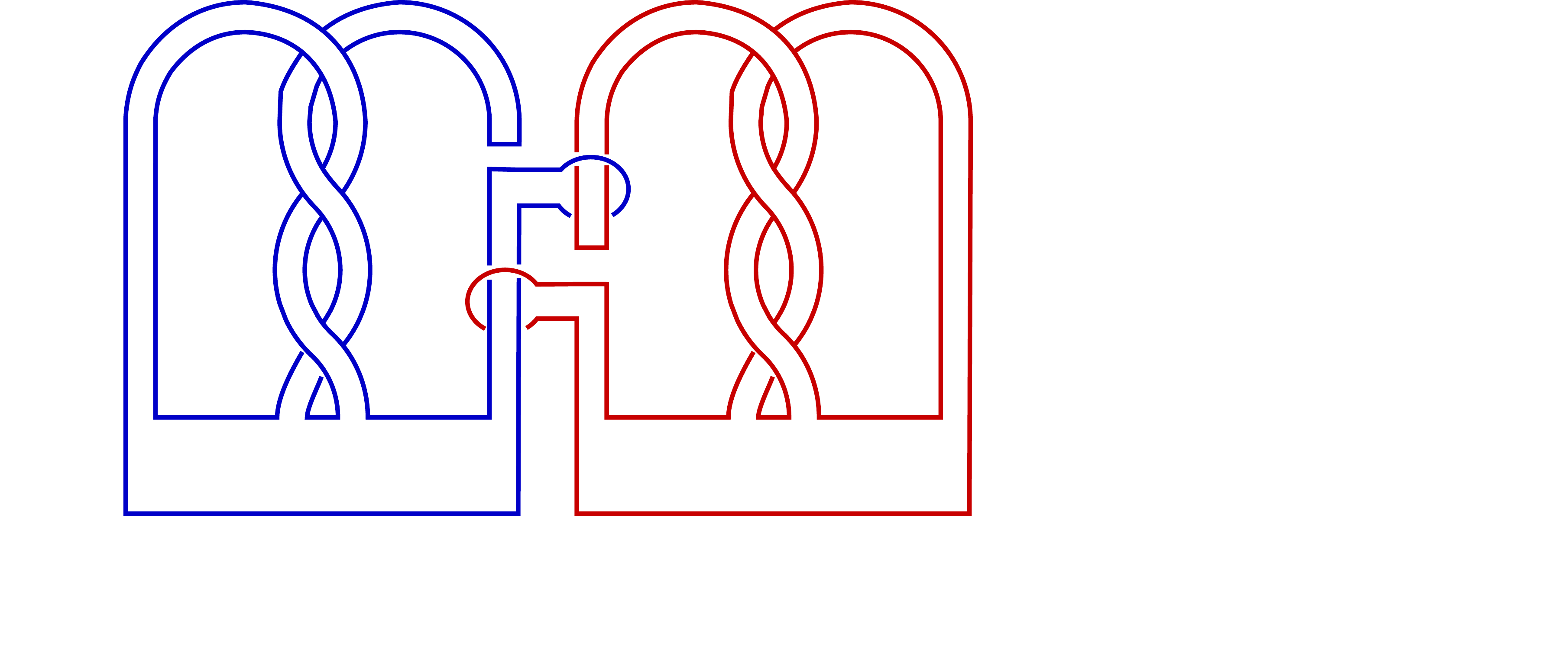}}
\put(145,0) {$(b)$}
\put(200,0){\includegraphics[height=.1\textheight]{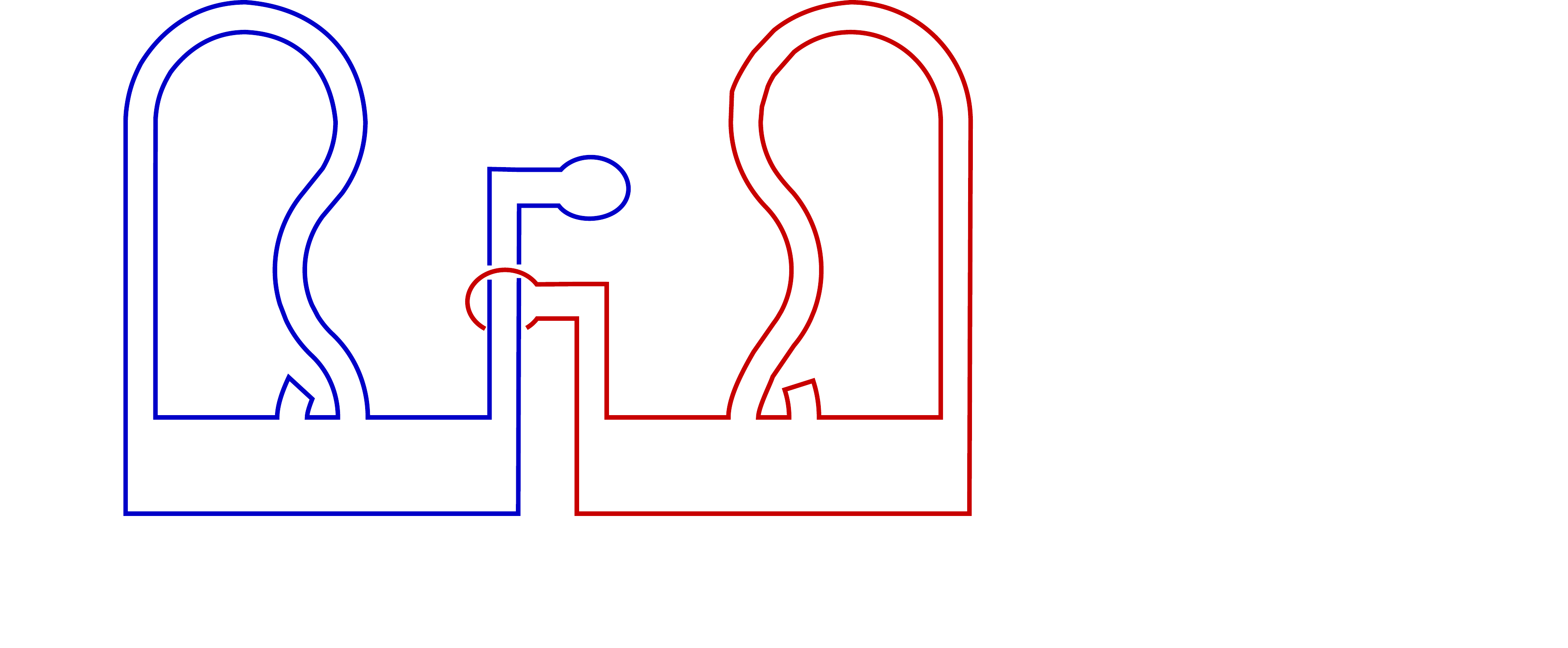}}
\put(245,0) {$(c)$}
\end{picture}
\caption{A $2$-component sublink $L_k\cup L_{k+1}$ of $L$ and a pair of band moves showing it is slice.} \label{fig:3-comp sublink}
\end{figure} 

Let $k \in \Z/n\Z$ and consider a 3-component sublink $L_{k-1}\cup L_k\cup L_{k+1}$. As in the proof of Theorem~\ref{thm: 2-component}, it is straightforward to verify that the classes of lifts of $L_{k-1}$ and $L_{k+1}$ generate $\A(L_k)$. By Corollary~\ref{cor:tool}, we conclude that $L$ is not concordant any link with the $k$th component unknotted. Again, the proof is complete by the symmetry of $L$.\end{proof}

Lastly, we prove Theorem~\ref{thm:alexanderpoly}.

\begin{figure}[h]
\begin{picture}(150,95)
\put(0,5){\includegraphics[height=.15\textheight]{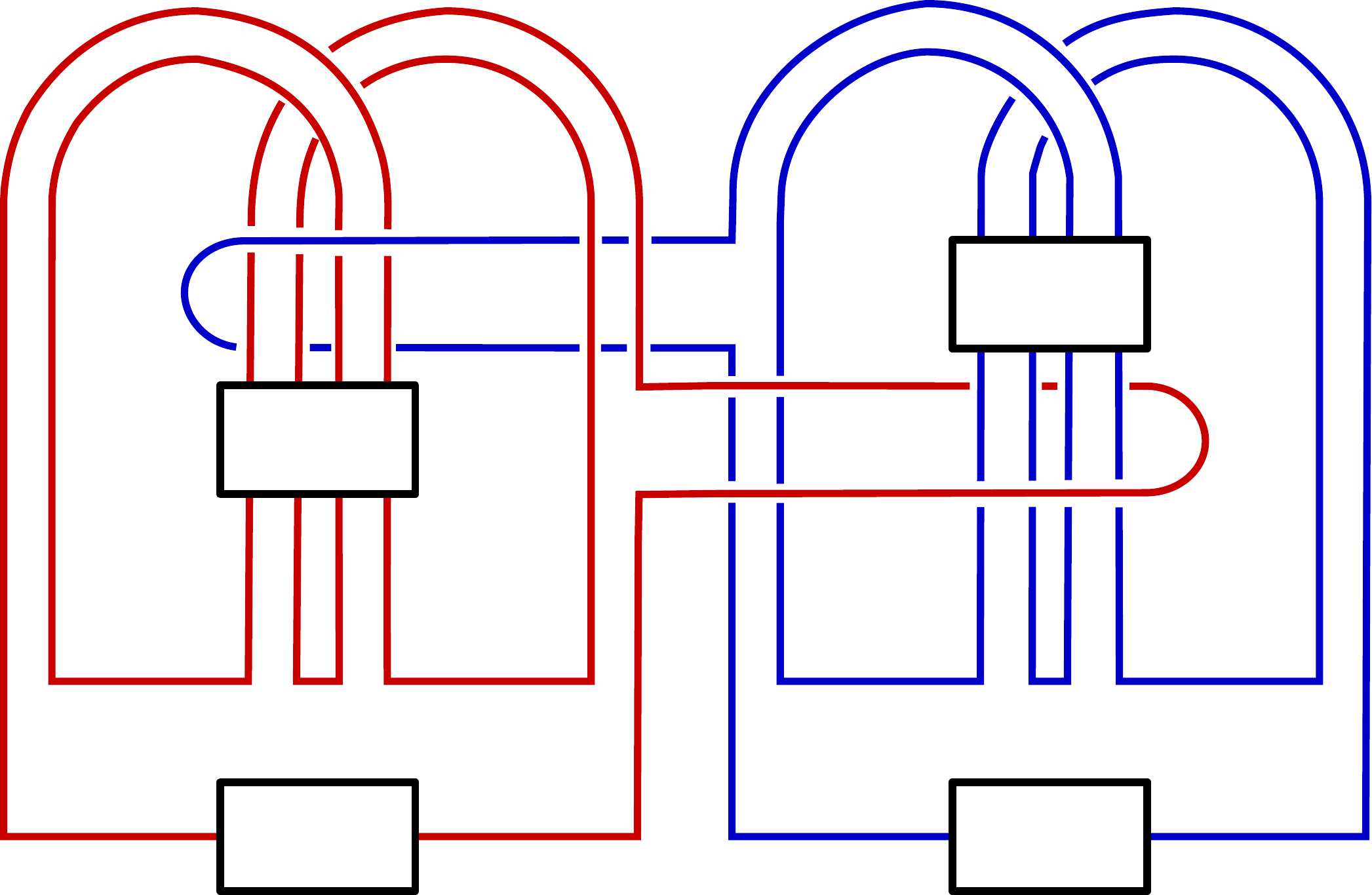}}
\put(28.5,50) {$m$}
\put(29.5,7.5) {$J$}
\put(104.5,65) {$m$}
\put(105,7.5) {$J$}
\end{picture}
\caption{A $2$-component link $L(m,J)=L_1(m,J)\cup L_2(m,J)$ of Theorem~\ref{thm:alexanderpoly}.  Each box containing an integer $m$ indicates the bands passing through the
box have $m$ full twists rather than all the strands. Each box containing a knot $J$ indicates the strand passing through the box is tied into $J$.}\label{fig:alexander}
\end{figure}
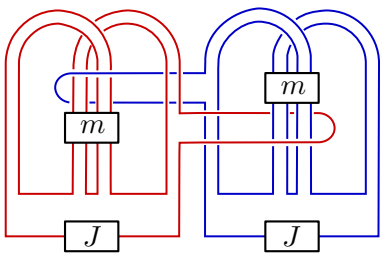 

\begin{proof}[Proof of Theorem~\ref{thm:alexanderpoly}.] Let $L=L(m,J)=L_1(m,J)\cup L_2(m,J)$ be the link of Figure~\ref{fig:alexander}. We choose $m$ large enough so that $\Delta_{L_1(m,U)}(t)=\Delta_{L_2(m,U)}(t)$ is relatively prime to every polynomial in the finite set $D$. Since $L_1(m,J)$ and $L_2(m,J)$ are isotopic to a knot obtained as the connected sum of  $J$ with a slice knot, the first condition of the theorem is satisfied.

Suppose $L$ is concordant to a link $L'=L_1' \cup L_2'$ where $\Delta_{L_1'}(t) \in D$ and let $-J$ be the knot obtained by taking the mirror image $J$ and reversing the orientation. By locally tying $-J$ into the concordance from $L$ to $L'$ and stacking a concordance from $L(m,J \#-J)$ to $L(m,U)$, we see that $L(m,U)$ is concordant to a link $L''=L_1'' \cup L_2''$ where $L_1''$ is isotopic to a connected sum of $L_1'$ with $-J$. In particular, $\Delta_{L_1''}(t) = \Delta_{L_1'}(t)\Delta_{J}(t)$. By the assumption, $\Delta_{L_1(m,U)}(t)$ and $\Delta_{L_1''}(t)$ are relatively prime. As in the proof of Theorem~\ref{thm: 2-component}, it is straightforward to verify that each component of $L(m,U)$ is slice and the class of the lift of $L_2(m,U)$ generate $\A(L_1(m,U))$. We get a contradiction by Corollary~\ref{cor:tool}. The proof is complete by applying the same argument for the second component.\end{proof}

\bibliographystyle{alpha}
\bibliography{biblio}
\end{document}